\documentclass[12pt,letterpaper]{amsart}

\usepackage{amsmath}
\usepackage{mathtools}
\usepackage{amssymb,todonotes}
\usepackage[pagebackref,colorlinks=true,linkcolor=blue,citecolor=blue]{hyperref}
\usepackage{tikz-cd}

\tikzset{
  symbol/.style={
    draw=none,
    every to/.append style={
      edge node={node [sloped, allow upside down, auto=false]{$#1$}}}
  }
}

\newcommand{\BC}{{\mathbb {C}}}

\newcommand{\Fh}{{\mathfrak {h}}}

\newcommand{\Z}{\mathbb{Z}}
\newcommand{\ZZ}{\mathbb{Z}}

\newcommand{\R}{\mathbb{R}}
\newcommand{\RR}{\mathbb{R}}

\newcommand{\OO}{{\mathcal O}}

\newcommand{\SL}{\mathrm{SL}}
\newcommand{\GL}{\mathrm{GL}}

\newcommand{\nr}{\mathrm{nr}}
\newcommand{\Irr}{\mathrm{Irr}}

\newcommand{\fpp}{\mathrm{FPP}}

\newcommand{\Ind}{\mathrm{Ind}}
\newcommand{\LLC}{\mathrm{LLC}}
\newcommand{\Lie}{\mathrm{Lie}}

\newcommand{\mfr}[1]{\mathfrak{#1}}

\newcommand{\half}[1]{\frac{#1}{2}}

\newcommand{\comment}[1]{}

\newcommand{\Fr}{{\mathrm{Fr}}}

\newtheorem{thm}{Theorem}[section]

\newtheorem{lemma}[thm]{Lemma}
\newtheorem{prop}[thm]{Proposition}
\newtheorem {conj}[thm]{Conjecture}

\newtheorem {ques/conj}[thm]{Question/Conjecture}

\newtheorem{remark}[thm]{Remark}

\newtheorem{exmp}[thm]{Example}

\newtheorem{claim}[thm]{Claim}

\newtheorem{rmk}[thm]{Remark}

\newtheorem{manualtheoreminner}{Desideratum}
\newenvironment{manualtheorem}[1]{%
  \IfBlankTF{#1}
    {\renewcommand{\themanualtheoreminner}{\unskip}}
    {\renewcommand\themanualtheoreminner{#1}}%
  \manualtheoreminner
}{\endmanualtheoreminner}

\newtheorem*{globalcond*}{Global Condition}
\newtheorem*{localcond*}{Local Condition}

\newtheorem*{globalconj*}{Global Conjecture}
\newtheorem*{localconj*}{Local Conjecture}

\newtheorem*{nonzero*}{Conjecture on the non-vanishing of the normalized intertwining operators}
\newtheorem*{holo*}{Conjecture on the holomorphicity of the normalized intertwining operators}

\DeclareMathOperator{\ad}{ad}
\DeclareMathOperator{\Ad}{Ad}

\DeclareMathOperator{\Hom}{Hom}
\DeclareMathOperator{\Gal}{Gal}

\numberwithin{equation}{section}

\let\oldbullet\bullet
\renewcommand{\bullet}{{\vcenter{\hbox{\tiny$\oldbullet$}}}}

\begin{document}

\title[The FPP Conjecture for $p$-adic Groups]{The FPP Conjecture for $p$-adic Groups}

\author{Dihua Jiang}
\address{School of Mathematics, University of Minnesota, Minneapolis, MN 55455, USA}
\email{dhjiang@math.umn.edu}

\author{Baiying Liu}
\address{Department of Mathematics\\
Purdue University\\
West Lafayette, IN, 47907, USA}
\email{liu2053@purdue.edu}

\author{Chi-Heng Lo}
\address{Department of Mathematics\\
National University of Singapore\\
119076, Singapore}
\email{{ch\_lo@nus.edu.sg}}

\author{Lucas Mason-Brown}
\address{Department of Mathematics\\
University of Texas at Austin\\
Austin, TX, 78712, USA}
\email{lucas.masonbrown@austin.utexas.edu}

\subjclass[2020]{Primary , ; Secondary , }

\thanks{The first-named author is partially supported by the Simons Grants: SFI-MPS-SFM-00005659 and 
SFI-MPS-TSM-00013449. The second-named author is partially supported by the NSF Grant DMS-1848058 and the Simons Foundation: Travel Support for Mathematicians. The fourth-named author is partially supported by NSF Grant DMS-2501977.
}
\date{\today}

\begin{abstract}
The FPP conjecture, proposed by J. Adams, S. Miller, and D. Vogan (\cite{Vog23}) and proved by D. Davis and L. Mason-Brown (\cite{DMB24}), imposes a strong upper bound on the infinitesimal character of a unitary representation of a real reductive group. In this paper, we formulate an analogous conjecture for $p$-adic groups. We prove our conjecture for pure rational forms assuming a version of the Local Langlands Correspondence.
\end{abstract}

\maketitle

\section{Introduction}

Let $k$ be a local field of characteristic zero and let $G$ be the $k$-points of a connected reductive algebraic group $\mathbf{G}$ defined over $k$. Let $\Pi_u(G)$ denote the set of equivalence classes of irreducible unitary $G$-representations. One of the basic unsolved problems in harmonic analysis and representation theory of $G$
is to classify $\Pi_u(G)$.

For $k$ Archimedean, J. Adams, S. Miller, and D. Vogan (\cite{Vog23}) proposed the {\sl FPP Conjecture}, which 
posits a strong upper bound on $\Pi_u(G)$.  Recall that every irreducible $G$-representation has an \emph{infinitesimal character}, which can be regarded as a $W$-orbit on the dual abstract Cartan $\Fh^*$ of $G$. A representation is said to have ``real infinitesimal character" if this $W$-orbit belongs to the canonical real form $\mathfrak{h}_{\RR}^*$ of $\mathfrak{h}^*$. The FPP conjecture of Adams-Miller-Vogan asserts that if $\pi$ is an irreducible unitary $G$-representation of real infinitesimal character, then one of the following holds:
\begin{itemize}
    \item[(1)] The infinitesimal character $\lambda$ of $\pi$ (regarded as a \emph{dominant} element of the real dual Cartan $\mathfrak{h}_{\RR}^*$) satisfies the inequalities
    $$\langle \lambda, \alpha^{\vee}\rangle \leq 1,$$
    for every simple co-root $\alpha^{\vee}$;
    \item[(2)]The representation $\pi$ is cohomologically induced (in the weakly good range) from an irreducible representation satisfying the bounds described in (1).
\end{itemize}
This conjecture was recently proved by Dougal Davis and the third-named author in \cite{DMB24}. 

In this paper we formulate an analogous conjecture in the non-Archimedean case, and we prove our conjecture in the case when $G$ is a pure rational form (i.e. pure inner to quasi-split). Our proof is conditional on the validity of the local Langlands correspondence, which is still a conjecture in general. However, our result is unconditional in many cases in which the local Langlands correspondence is known, including when $G$ is a classical group. 

To state our conjecture, we will need to recall some preliminaries on the local Langlands correspondence. Assume henceforth that $k$ is non-Archimedean and let $\Pi(G)$ denote the admissible dual of $G$, i.e. the set of equivalence classes of irreducible smooth $G$-representations.  By passing to smooth vectors, we can (and will) regard $\Pi_u(G)$ as a subset of $\Pi(G)$. Fix an algebraic closure $\bar{k}$ of $k$ with Galois group $\Gamma=\mathrm{Gal}(\bar{k}/k)$. Let $W_k$ denote the Weil group of $k$, and let $^LG = G^{\vee}\rtimes \Gamma$ denote the $L$-group of $G$. Recall that an $L$-parameter for $^LG$ is a continuous homomorphism
$$\phi: W_k \times \SL(2,\mathbb{C}) \to {^LG}$$
such that
\begin{itemize}
    \item[(i)] $\phi$ respects the maps from $W_k$ and $^LG$ to $\Gamma$;
    \item[(ii)] $\phi(W_k)$ consists of semisimple elements;
    \item[(iii)] The restriction of $\phi$ to $\SL(2,\mathbb{C})$ is holomorphic.
\end{itemize}
We write $\Phi(^LG)$ for the set of $G^{\vee}$-conjugacy classes of $L$-parameters for $^LG$. The conjectural Local Langlands Correspondence is a map
$$\mathrm{LLC}: \Pi(G) \to \Phi(^LG)$$
satisfying various properties, see \S \ref{sec: desiderata LLC} and \S \ref{sec: KL} below. 

According to \cite{SZ18}, the set $\Phi(^LG)$ is in bijection with $G$-conjugacy classes of triples $(P,\phi_0,\nu)$, where $P$ is a parabolic subgroup of $G$ with Levi decomposition $P=MN$, $\phi_0$ is a tempered (i.e. bounded) $L$-parameter for $^LM$, and $\nu$ is positive real-valued character of $M$ which is dominant with respect to $P$. The element $\nu$ is called the ``exponent" of the $L$-parameter and is denoted by $\nu(\phi)$. We can regard $\nu(\phi)$ as a dominant element of the dual real Cartan $\mathfrak{h}^*_{\mathbb{R}}$.

The \emph{infinitesimal character} of an $L$-parameter $\phi$ is the map $\lambda_{\phi}: W_k \to {}^LG$ defined by
$$\lambda_{\phi}(w) = \phi \left(w, \begin{pmatrix}|w|^{\frac{1}{2}} & 0 \\0 & |w|^{-\frac{1}{2}}\end{pmatrix}\right).$$
We can regard $\lambda_{\phi}$ as an $L$-parameter by composing with the projection $W_k \times \SL(2,\mathbb{C}) \to W_k$. In this way, we can define the exponent $\nu(\lambda_{\phi})$ of $\lambda_\phi$. Now, our version of the FPP conjecture can be stated as follows.

\begin{conj}[FPP conjecture for $p$-adic groups]\label{conj:FPP}
Assume that $\pi \in \Pi_u(G)$ is an irreducible unitary $G$-representation and let $\phi = \mathrm{LLC}(\pi)$. Then the element $\nu(\lambda_{\phi}) \in \mathfrak{h}_{\RR}^*$ satisfies
$$\langle \nu(\lambda_{\phi}), \alpha^{\vee}\rangle \leq 1,$$
for every simple co-root $\alpha^{\vee}$.
\end{conj}


\begin{rmk}
The set 
$$\{x \in \mathfrak{h}_{\mathbb{R}}^* \mid 0 \leq \langle  x,\alpha^{\vee}\rangle \leq 1, \text{ for all simple co-roots } \alpha^{\vee}\}$$
is called the ``fundamental parallelepiped" (FPP) of $G$. Our conjecture asserts that if $\pi$ is unitary, then $\nu(\lambda_{\mathrm{LLC}(\pi)})$ belongs to this set. 
\end{rmk}

Our main result is a proof of Conjecture \ref{conj:FPP} in the case when $G$ is a pure rational form, assuming a version of the local Langlands correspondence.

\begin{thm}\label{main thm intro}
Assume that $G$ is a pure rational form, and that there exists a map $\mathrm{LLC}$ compatible with the Langlands classification and satisfying the Kazhdan-Lusztig hypothesis (see \S \ref{sec: desiderata LLC} and \S \ref{sec: KL}).
Then Conjecture \ref{conj:FPP} is true.
\end{thm}


In the Archimedean case, the FPP conjecture was proved in \cite{DMB24} using techniques from Hodge theory. In the non-Archimedean case, such techniques are not available. 
To prove Theorem \ref{main thm intro}, we instead show that any irreducible Hermitian representation $\pi$ lying `outside' the FPP must belong to an \emph{unbounded} analytic family of irreducible Hermitian representations. We then appeal to general facts about the topology of the unitary dual to deduce that $\pi$ is not unitary. 

The idea of proving non-unitarity by constructing an unbounded family of irreducibly induced Hermitian representations has been applied previously in the literature, e.g. in \cite{Tad09}. For symplectic and odd special orthogonal groups, Tadi{\'c} established an irreducibility criterion (\cite[Proposition 3.2]{Tad09}) in terms of supercuspidal support, and used it to deduce an upper bound on the unitary dual (\cite[Proposition 4.6]{Tad09}). In the cases in which it applies, Tadi{\'c}'s bound is better than ours. However, it does not seem possible to generalize Tadi{\'c}'s result beyond classical groups. 
The novelty of our result is that it establishes an irreducibility criterion (Theorem \ref{thm ind irred}) assuming only the existence of a Local Langlands Correspondence satisfying several desiderata. Our proof makes use of the geometry of Vogan varieties, and is uniform for all pure rational forms.

We conclude with several remarks regarding the significance of our main theorem and its relation to previous work. First, we remark that our bound on $\Pi_u(G)$ is \emph{sharp}, in the following sense: for every group $G$, there are unitary representations for which all of the inequalities in Conjecture \ref{conj:FPP} are equalities, for example the trivial and Steinberg representations. This suggests that it may be possible to use the FPP bound as a replacement for the unitarity bound in various applications. Second, the FPP conjecture provides a setting in which arithmetic information extracted from the local $L$-parameter (i.e. the exponent of the infinitesimal character) controls the harmonic analysis on $G$ (i.e. the structure of unitary $G$-representations). Among other such examples, let us mention the 
local {\sl Gan-Gross-Prasad Conjecture} (\cite{GGP12}) and the local {\sl Wavefront Set Conjecture} of Jiang-Liu-Zhang (\cite{JLZ22}). Third, we note that in many cases, the required form of the local Langlands correspondence is known; in such cases, Theorem \ref{main thm intro} is true unconditionally. For example, this is true in the following cases, see \cite[Theorem 5.4]{Solleveld2025}:
\begin{itemize}
    \item all representations of classical groups (i.e. symplectic, (special) orthogonal, unitary, and general (s)pin groups),
    \item principal series representations of quasi-split groups, and
    \item unipotent representations of arbitrary groups.
\end{itemize}

As explained above, our main result provides a uniform upper bound on the unitary dual of any $p$-adic reductive group, conditional on the validity of the Langlands correspondence. To our knowledge, ours is the only result of this type. However, there are various special cases (i.e. special classes of groups and types of representations) for which the unitary dual is known completely. In these cases, the FPP bound can be verified directly. Below, we give a summary of previous such results. Using Hecke algebra techniques, Barbasch, Ciubotaru, Moy, and Pantano classified the spherical unitary representations of split $p$-adic groups in \cite{BM89, BM93, Ciu05, BC09, Bar10}; see also \cite{BCP08} for an expository survey. The full unitary dual has also been classified for the following classes of groups: $\GL_n(k)$ (\cite{Tad86}), $G_2(k)$ (\cite{Mui97}), and general linear groups over central division algebras (\cite{BR04, Sec09}). For quasisplit classical groups, the Whittaker-generic unitary representations were classified in \cite{LMT04}.

\subsection{Acknowledgments}
The authors would like to thank Freydoon Shahidi for constant support. The authors also would like to thank Jeff Adams and David Vogan for helpful discussions, comments and suggestions.

\section{Preliminaries on the local Langlands correspondence}

In this section, we introduce notation and preliminaries on the Langlands classification and the Kazhdan-Lusztig Hypothesis.

\subsection{Notation and preliminaries}\label{sec notation}

Let $k$ be a non-Archimedean local field of characteristic $0$ with residue field $\mathbb{F}_q$. Fix an algebraic closure $\bar{k}$ of $k$ with Galois group $\Gamma=\mathrm{Gal}(\bar{k}/k)$. Let $\mathbf{G}$ be a connected reductive algebraic group defined over $k$, and let $G= \mathbf{G}(k)$ denote the $k$-point of $\mathbf{G}$. Choose a maximal $k$-split torus $\mathbf{A}_0 \subset \mathbf{G}$ and a maximal $k$-torus $\mathbf{T}_0 \subset \mathbf{G}$ containing $\mathbf{A}_0$. Let $\mathbf{P}_0 \subset \mathbf{G}$ be a minimal $k$-parabolic subgroup containing $\mathbf{T}_0$ and let $\mathbf{P}_0=\mathbf{M}_0\mathbf{N}_0$ denote the Levi decomposition corresponding to $\mathbf{T}_0$. Let $R=R(G)$ denote the roots of $\mathbf{T}_0$ on $\mathbf{G}$ and choose a positive system $R^+=R^+(G)$ compatible with $\mathbf{P}_0$. Denote the corresponding set of simple roots by $\Delta=\Delta(G)$. 
 
Let $(G^{\vee},H^{\vee},B^{\vee},\{X_{\alpha}\})$ denote the pinned complex dual group associated to $\mathbf{G}$. Then $\Gamma$ acts on $G^{\vee}$ by pinned automorphisms. The $L$-group of $\mathbf{G}$ is $^LG = G^{\vee} \rtimes \Gamma$. We will need to recall some elementary facts about parabolic subgroups of $G$ and ${}^L G$ following \cite[\S 3]{Bor79}. 

Let $p(\mathbf{G}/k)$ denote the set of $G$-conjugacy classes of $k$-parabolic subgroups of $\mathbf{G}$ and let $p(\mathbf{G})_k$ denote the set of $\Gamma$-stable $G$-conjugacy classes of parabolic subgroups of $\mathbf{G}$. Note that $p(\mathbf{G}/k)$  is naturally a subset of $p(\mathbf{G})_k$ and $p(\mathbf{G})_k=p(\mathbf{G}/k)$ if $\mathbf{G}$ is quasi-split. There is a canonical bijection between $p(\mathbf{G})_k$ and the set of $\Gamma$-stable subsets of $\Delta$. Under this bijection $p(\mathbf{G}/k)$ corresponds to the set of $\Gamma$-stable subsets containing the simple roots in $\mathbf{M}_0$. A $k$-parabolic subgroup of $\mathbf{G}$ is said to be \emph{standard} if it contains $\mathbf{P}_0$. Every conjugacy class in $p(\mathbf{G}/k)$ contains a unique standard representative.

A closed subgroup of $^LG$ is called ``parabolic" if it is the normalizer of a parabolic subgroup of $G^{\vee}$ and meets every connected component of $^LG$. Let $p({}^L G)$ denote the set of $G^{\vee}$-conjugacy classes of parabolic subgroups of ${}^L G$. There is a canonical bijection between $p({}^L G)$ and $p(\mathbf{G})_k$. We say that a parabolic subgroup of ${}^L G$ is \emph{relevant} if it corresponds under this bijection to an element of $p(\mathbf{G}/k)$. A parabolic subgroup of $^LG$ is said to be standard if it contains $B^{\vee}$. Every conjugacy class in $p({}^L G)$ contains a unique standard representative.

\subsection{Langlands classification of irreducible representations}\label{sec: Langlands classification}

Let $\Pi(G)$ (resp. $\Pi_t(G)$) denote the set of equivalence classes of irreducible smooth (resp. irreducible smooth tempered) $G$-representations. We will now recall the statement of the Langlands classification for $p$-adic groups, which reduces the classification of $\Pi(G)$ to tempered representations of Levi subgroups of $G$.

For any $k$-Levi subgroup $\mathbf{M} \subset \mathbf{G}$, let $X^*(M) = X^*(\mathbf{M})_k$ denote the abelian group of rational characters $\mathbf{M} \to \GL_1$ defined over $k$ and let
$$\mathfrak{a}_M^* := X^*(M) \otimes_{\ZZ} \RR.$$
Note that there is a canonical isomorphism
\begin{equation}\label{eq:astar}\mathfrak{a}_M^* \simeq \Hom(M,\RR_{>0}).\end{equation}
see \cite[(1.2.2)]{Sil78}.

A \emph{standard triple} for $\mathbf{G}$ is a triple of the form $(P,\pi_t,\nu)$, where 
\begin{itemize}
    \item $P$ is the $k$-points of a standard $k$-parabolic subgroup $\mathbf{P}$ with Levi decomposition $\mathbf{P}=\mathbf{M}\mathbf{N}$; 
    \item $\pi_t$ is an element of $\Pi_t(M)$;
    \item $\nu$ is an element of $\mathfrak{a}_M^*$ satisfying the dominance condition
    $$\langle \nu, \alpha^{\vee} \rangle > 0, \qquad \forall \alpha \in \Delta(G) \setminus \Delta(M).$$
\end{itemize}
Using the isomorphism (\ref{eq:astar}) we can also regard $\nu$ as a positive real-valued character of $M$.

\begin{thm}[{\cite{Sil78, BW80, Kon03}}]\label{thm:Langlandsreps}
Suppose $(P,\pi_t,\nu)$ is a standard triple for $\mathbf{G}$. Then the induced representation $\Ind^G_P (\pi_t \otimes \nu)$ has a unique irreducible quotient $\pi(P,\pi_t,\nu)$. Moreover $(P,\pi_t,\nu) \mapsto \pi(P,\pi_t,\nu)$ defines a bijection between the set of standard triples and $\Pi(G)$.
\end{thm}

We will use the notation
$$\pi \leftrightarrow (P,\pi_t,\nu)$$
to indicate that $\pi$ corresponds to the triple $(P,\pi_t,\nu)$ under the bijection in Theorem \ref{thm:Langlandsreps}.

\subsection{Langlands classification of parameters}\label{lc parameter}

Recall that an $L$-parameter {for $^LG$} is a continuous homomorphism
$$\phi: W_k \times \SL_2(\BC) \to {}^LG,$$
such that
\begin{itemize}
    \item[(i)] $\phi$ respects the maps from $W_k$ and $^LG$ to $\Gamma$;
    \item[(ii)] $\phi(W_k)$ consists of semisimple elements;
    \item[(iii)] The restriction of $\phi$ to $\SL_2(\BC)$ is holomorphic.
\end{itemize}

We say that $\phi$ is \emph{tempered} if the projection to $G^{\vee}$ of $\mathrm{Im}(\phi)$ is bounded. Define
\begin{align*}
P(^LG) &= \{\text{$L$-parameters $\phi: W_k \times \SL_2(\mathbb{C}) \to {}^LG$}\},\\
P_t(^LG) &= \{\text{tempered $L$-parameters $\phi: W_k \times \SL_2(\mathbb{C}) \to {}^LG$}\}.
\end{align*}
Note that $G^{\vee}$ acts by conjugation on $P(^LG)$ and preserves the subset $P_t(^LG)$. We denote
\begin{align*}
\Phi(^LG) &= P(^LG)/\sim,\\
\Phi_t(^LG) &= P_t(^LG)/\sim.
\end{align*}
For $\phi \in P(^LG)$, we write $[\phi]$ for the $G^{\vee}$-conjugacy class of $\phi$.

Let $S_{\phi}$ denote the finite component group of the centralizer of $\phi$, i.e. 
\[S_{\phi} = Z_{G^{\vee}}(\phi)/ Z_{G^{\vee}}(\phi)^{\circ}.\] 
A (tempered) \emph{enhanced $L$-parameter} is a pair $(\phi,\tau)$ consisting of a (tempered) $L$-parameter $\phi$ and an irreducible representation $\tau$ of $S_{\phi}$. Define
\begin{align*}
P^{\mfr{e}}(^LG) &= \{\text{enhanced $L$-parameters $(\phi,\tau)$}\},\\
P_t^{\mfr{e}}(^LG) &= \{\text{tempered enhanced $L$-parameters $(\phi,\tau)$}\}.
\end{align*}
Again, $G^{\vee}$ acts by conjugation on $P^{\mfr{e}}(^LG)$ and preserves the subset $P_t^{\mfr{e}}(^LG)$. We denote
\begin{align*}
\Phi^{\mfr{e}}(^LG) &= P^{\mfr{e}}(^LG)/\sim,\\
\Phi_t^{\mfr{e}}(^LG) &= P_t^{\mfr{e}}(^LG)/\sim.
\end{align*}

For any $k$-Levi subgroup $\mathbf{M} \subset \mathbf{G}$, let $Z({}^L M)^0_h$ denote the subgroup of the complex torus $Z({}^L M)^0$ consisting of hyperbolic elements. Then there is a natural identification
$$\mathfrak{a}_M^* \simeq Z({}^LM)^0_h, \qquad \nu \leftrightarrow z(\nu),$$
see \cite[Corollary 4.3]{SZ18}. If $z \in Z({}^L M)^0$ and $\phi \in P(^LM)$, we define the \emph{$z$-twist} of $\phi$ by
$$\phi^z: W_k \times \SL_2(\BC) \to {}^LM, \qquad \phi^z(w,x) = \phi(w,x)z^{d(w)},$$
where $d(w)$ is defined by $|w|=q^{d(w)}$.
This is again an $L$-parameter.

A \emph{standard triple} for $^LG$ is a triple of the form $(P,[\phi_t],\nu)$, where $P$ and $\nu$ are as in \S \ref{sec: Langlands classification}, and $[\phi_t] \in \Phi_t({}^L M)$. An \emph{enhanced standard triple} for $G$ is a triple of the form $(P,([\phi_t],\tau),\nu)$, where $([\phi_t],\tau) \in \Phi_t^{\mfr{e}}(^LM)$.

\begin{thm}[{ \cite[Theorem 4.6]{SZ18}}]\label{thm:Langlandsparameters}
The following are true:
\begin{itemize}
    \item[(i)] Suppose $(P,[\phi_t],\nu)$ is a standard triple for $^LG$. Then $({}^LM \hookrightarrow {}^LG) \circ \phi_t^{z(\nu)}$ is an $L$-parameter for $^LG$. Moreover, $(P,[\phi_t],\nu) \mapsto ({}^LM \hookrightarrow {}^LG) \circ [\phi_t^{z(\nu)}]$ defines a bijection between the set of standard triples for ${}^L G$ and $\Phi(^LG)$.
    \item[(ii)] If $\phi \in \Phi(^LG)$ corresponds to $(P,[\phi_t],\nu)$ under the bijection in (i), then the natural group homomorphism $S_{\phi_t} \to S_{\phi}$ is an isomorphism. Thus, the bijection in (i) lifts to a bijection between the set of enhanced standard triples for $^LG$ and $\Phi^{\mfr{e}}(^LG)$.
\end{itemize}
\end{thm}

We use the notation $[\phi] \leftrightarrow (P,[\phi_t],\nu)$ (resp. $([\phi],\tau) \leftrightarrow (P,([\phi_t],\tau_t),\nu)$ to indicate that $[\phi]$ (resp. $([\phi],\tau)$) corresponds to the triple $(P,[\phi_t],\nu)$ (resp. $(P,([\phi_t],\tau_t),\nu)$) under the bijections in Theorem \ref{thm:Langlandsparameters}. Finally, we recall certain facts from the proof of Part (i) of Theorem \ref{thm:Langlandsparameters}.

\begin{prop}[{\cite[Proposition 5.3(i), 5.5]{SZ18}}]\label{prop SZ18}
    Suppose that $[\phi] \leftrightarrow (P=MN,[\phi_t],\nu)$ and take a representative $\phi$. Then the hyperbolic part $z_{\phi}:=\phi(\Fr)$ is independent of the choice of $\Fr$. Moreover, let ${}^L P_{\phi}={}^L M_{\phi}{}^L N_{\phi}$ be the parabolic subgroup (not necessarily standard) determined by $z_{\phi}$. Namely, ${}^L M_{\phi}= Z_{{}^L G}(z_{\phi})$, and $\textrm{Lie}(N_{\phi})= \{x \in \mfr{g}^{\vee}\ | \ \Ad(z_{\phi})x= tx \text{ for some }t>1\}.$ Then the triple $({}^L M_{\phi},{}^L N_{\phi}, z_{\phi})$ is conjugate to $({}^LM, {}^L N, z(\nu))$ under $G^{\vee}$.
\end{prop}

\subsection{Local Langlands Correspondence}\label{sec: desiderata LLC}

In this subsection, we recall the general form of the Local Langlands Correspondence. We will assume in the subsection that $\mathbf{G}$ is quasi-split. Recall that the pure inner twists of $\mathbf{G}$ are parameterized by $H^1(k,\mathbf{G})$. We write ${}^{\delta}\mathbf{G}$ for the inner twist parameterized by $\delta \in H^1(k,\mathbf{G})$ and denote the $k$-points of ${}^{\delta}\mathbf{G}$ by ${}^{\delta}\mathbf{G}(k) = {}^{\delta}G$ (note that $G={}^0G$ is the quasi-split form). Define 
$$\Pi^{\mathrm{pure}}(\mathbf{G}/k) := \bigsqcup_{\delta \in H^1(k,\mathbf{G})} \Pi({}^{\delta}G).$$
Recall the Kottwitz isomorphism $H^{1}(k, \textbf{G}) \cong \Hom( \pi_0(Z({G}^{\vee})^{\Gamma}), \BC^{\times})$ (\cite{Kot84}). For $\delta \in H^{1}(k, \textbf{G})$, let $\chi_{\delta}$ denote the corresponding character of $ \pi_0(Z({G}^{\vee})^{\Gamma})$, and let $\Irr( S_{\phi} , \delta)$ denote the set of irreducible representations $\tau$ of $S_{\phi}$ such that $\tau|_{\pi_0(Z({G}^{\vee})^{\Gamma})}$ is a direct sum of characters isomorphic to $\chi_{\delta}$. We have the following lemma.  

\begin{lemma}[{\cite[Lemma 5.7]{Kal16}}]\label{lem:relevant}
Let $\phi$ be an $L$-parameter for ${}^L G$ and let $\delta \in H^1(k, \textbf{G})$. Then $\Irr( S_{\phi} , \delta)$ is non-empty if and only if every parabolic subgroup of $^LG$ containing $\mathrm{Im}(\phi)$ is ${}^\delta \mathbf{G}$-relevant (see the end of Section \ref{sec notation}).
 \end{lemma}

Any parameter $\phi \in \Phi(^LG)$ satisfying the equivalent conditions of Lemma \ref{lem:relevant} is called ${}^{\delta}\mathbf{G}$\emph{-relevant}.

\vspace{5mm}

By a \emph{Local Langlands Correspondence} we mean a collection of bijections
$$
    \left(\LLC_{\mathbf{M}}^{\mfr{e}}: \Pi^{\mathrm{pure}}(\mathbf{M}/k) \xrightarrow{\sim} \Phi^{\mfr{e}}({}^L M)\right)_{\mathbf{M}}
$$
indexed by $k$-Levi subgroups $\mathbf{M}$ of $\mathbf{G}$. For each $\pi \in \Pi({}^\delta M)$ such that $\LLC_{\mathbf{M}}^{\mfr{e}}(\pi) = (\phi,\tau)$, we require that $\tau \in \mathrm{Irr}(S_{\phi},\delta)$.

The main additional condition we will impose is that the bijections $\LLC_{\mathbf{M}}^{\mfr{e}}$ are compatible with the Langlands classifications of representations (see Theorem \ref{thm:Langlandsreps}) and parameters (see Theorem \ref{thm:Langlandsparameters}) in the following sense:

\begin{manualtheorem}{(LC)}\label{LC}
A local Langlands Correspondence $(\LLC_{\mathbf{M}}^{\mfr{e}})_{\textbf{M}}$ is compatible with Langlands classification if for any irreducible representation $\pi \leftrightarrow (P, \pi_t, \nu)$, we have $\LLC_{\textbf{G}}^{\mfr{e}}(\pi)\leftrightarrow (P, \LLC_{\textbf{M}}^{\mfr{e}}(\pi_t), \nu)$.
\end{manualtheorem}

We write $\LLC_{\mathbf{M}}$ for the first coordinate of $\LLC_{\mathbf{M}}^{\mfr{e}}$. Thus $\LLC_{\mathbf{M}}$ is a surjection from $\Pi^{\mathrm{pure}}(M)$ onto $\Phi(^LM)$. For $[\phi] \in \Phi(^LM)$, we write
$$\Pi_{\phi}^{\mathrm{pure}}(\mathbf{M}/k) := \mathrm{LLC}_{\mathbf{M}}^{-1}(\phi)$$
for the Vogan $L$-packet associated to $\phi$.


\subsection{Kazhdan-Lusztig Hypothesis}\label{sec: KL}

In order to prove Conjecture \ref{conj:FPP}, we also need to assume that the Local Langlands Correspondence satisfies one additional property called the \emph{Kazhdan-Lusztig hypothesis}.

Recall that an infinitesimal character for $^LG$ is a continuous homomorphism
$$\lambda: W_k \to {}^LG$$
such that
\begin{itemize}
    \item[(i)] $\lambda$ respects the maps from $W_k$ and ${}^LG$ to $\Gamma$;
    \item[(ii)] $\phi(W_k)$ consists of semisimple elements.
\end{itemize}
To each $L$-parameter $\phi$, we can associate an infinitesimal character $\lambda_{\phi}$ according to the formula
$$\lambda_{\phi}(w) = \phi \left(w, \begin{pmatrix}|w|^{\frac{1}{2}} & 0 \\0 & |w|^{-\frac{1}{2}}\end{pmatrix}\right).$$
For the remainder of this subsection, we fix an infinitesimal character $\lambda: W_k \to {}^LG$ and a Local Langlands Correspondence $(\LLC^{\mfr{e}}_{\mathbf{M}})_{\mathbf{M}}$ for $G$. Define the sets of parameters
\begin{align*}
P_{\lambda}(^LG) &= \{\phi \in P(^LG) \mid \lambda_{\phi} = \lambda\},\\
P^{\mfr{e}}_{\lambda}(^LG) &= \{(\phi,\tau) \in P^{\mfr{e}}(^LG) \mid \lambda_{\phi} = \lambda\},\\
\Phi_{\lambda}(^LG) &= P_{\lambda}(^LG)/\sim,\\
\Phi^{\mfr{e}}_{\lambda}(^LG) &= P^{\mfr{e}}_{\lambda}(^LG)/\sim,
\end{align*}
and the set of representations
$$\Pi^{\mathrm{pure}}_{\lambda}(\mathbf{G}/k) = \bigsqcup_{[\phi] \in \Phi_{\lambda}(^LG)} \Pi_{\phi}^{\mathrm{pure}}(\mathbf{G}/k).$$
The \emph{Vogan variety} associated to $\lambda$ is a finite-dimensional complex vector space
$$V_{\lambda} := \{x \in \mathfrak{g}^{\vee} \mid \Ad(\lambda(w))x = |w|x, \ \forall w \in W_k\}.$$
The reductive group $H_{\lambda} := Z_{G^{\vee}}(\lambda(W_k))$ acts on $V_{\lambda}$ with finitely many orbits, see \cite[Proposition 4.5]{Vog93}.

There is a surjective $H^{\vee}_{\lambda}$-equivariant map
\begin{align}\label{eq Vogan variety orbit}
    P_{\lambda}(^LG) \twoheadrightarrow V_{\lambda}, \qquad  \phi \mapsto x_{\phi} :=d\phi\left(0,\begin{pmatrix} 0 & 1\\0 & 0\end{pmatrix}\right).
\end{align}
This map induces a bijection on the level of $H_{\lambda}$-orbits. Moreover, for any $\phi \in P_{\lambda}(^LG)$, the natural group homomorphism
\begin{align}\label{eq maps of component groups}
    S_{\phi} = \pi_0(Z_{G^{\vee}}(\phi)) = \pi_0(Z_{H_{\lambda}}(\phi)) \to \pi_0(Z_{H_{\lambda}}(x_{\phi}))
\end{align}
is an isomorphism. Thus, we can identify the set $\Phi^{\mfr{e}}_{\lambda}(^LG)$ of enhanced $L$-parameters with the set of pairs $(\OO, \mathcal{L})$ consisting of an $H_{\lambda}$-orbit $\OO \subset V_{\lambda}$ and an irreducible $H_{\lambda}$-equivariant local system on $\OO$. To each such pair (and hence to each enhanced parameter $([\phi],\tau) \in \Phi^{\mfr{e}}_{\lambda}(^LG)$), we can associate two classes in the Grothendieck group $K^{H_{\lambda}}(V_{\lambda})$ of $H_{\lambda}$-equivariant constructible sheaves on $V_{\lambda}$, as follows. Let $j: \OO \hookrightarrow V_{\lambda}$ denote the locally-closed embedding. The first class is defined by
$$\mu([\phi],\tau) := (-1)^{\dim(\OO)} [j_!\mathcal{L}],$$
where $j_!$ is `extension by zero' and $[\bullet]$ denotes the class in $K^{H_{\lambda}}(V_{\lambda})$. The second class is defined by
$$P([\phi],\tau) := \sum_i (-1)^i H^i(j_{!*}\mathcal{L}[-\dim(\OO)]),$$
where $j_{!*}$ denotes the intermediate extension of the simple perverse sheaf $\mathcal{L}[-\dim(\OO)]$. 

It is clear that both sets $\{\mu([\phi],\tau)\}$ and $\{P([\phi],\tau)\}$ are $\ZZ$-bases for $K^{{H_{\lambda}}}(V_{\lambda})$. We write $m_g$ for the change-of-basis matrix, i.e.
\begin{equation}\label{eq:mg}\mu([\phi],\tau) = \sum_{([\phi]',\tau') \in \Phi^{\mfr{e}}_{\lambda}(^LG)} m_g(([\phi],\tau),([\phi]',\tau')) P([\phi]',\tau').\end{equation}
Via the Local Langlands Correspondence, we can also associate to $([\phi],\tau) \in \Phi^{\mfr{e}}_{\lambda}({}^LG)$ two different elements of the free $\ZZ$-module $\ZZ \Pi_{\lambda}^{\mathrm{pure}}(\mathbf{G}/k)$, namely the irreducible representation $\pi([\phi],\tau) = (\mathrm{LLC}_G^{\mfr{e}})^{-1}([\phi],\tau)$ and the standard representation $M([\phi],\tau)$ with Langlands quotient $\pi([
\phi],\tau)$. Both sets $\{\pi([\phi],\tau)\}$ and $\{M([\phi],\tau)\}$ are $\ZZ$-bases for $\ZZ \Pi_{\lambda}^{\mathrm{pure}}(\mathbf{G}/k)$. We write $m_r$ for the change-of-basis matrix, i.e.
\begin{equation}\label{eq:mr}M([\phi],\tau) = \sum_{([\phi]',\tau') \in \Phi^{\mfr{e}}_{\lambda}(^LG)} m_r(([\phi],\tau),([\phi]',\tau')) \pi([\phi]',\tau').
\end{equation}
It is natural to expect that the matrices $m_g$ and $m_r$ are mutually inverse transpose (up to signs). This expectation is called the `Kazhdan-Lusztig hypothesis'.

\begin{manualtheorem}{(KL)}\label{KL}
    A Local Langlands Correspondence $(\LLC^{\mfr{e}}_{\mathbf{M}})_{\mathbf{M}}$ satisfies the \emph{Kazhdan-Lusztig hypothesis} if for every $k$-Levi subgroup $\mathbf{M}$ of $\mathbf{G}$, any infinitesimal character $\lambda$ for ${}^L M$, and any pair of enhanced parameters $([\phi],\tau),([\phi]',\tau') \in \Phi^{\mfr{e}}_{\lambda}(^LM)$, there is an equality
$$m_r(([\phi],\tau),([\phi]',\tau')) = (-1)^{\dim \OO - \dim \OO'} m_g(([\phi]',\tau'),([\phi],\tau)),$$
where $\OO$ (resp. $\OO'$) is the $H_{\lambda}$-orbit corresponding to $[\phi]$ (resp. $[\phi']$).
\end{manualtheorem}

\section{FPP Conjecture}

In this section, we prove our main result, which we restate here for convenience.

\begin{thm}\label{thm main}
Suppose that $G$ is a pure rational form and assume that 
there exists a local Langlands correspondence $(\LLC_{\mathbf{M}}^{\mfr{e}})_{\mathbf{M}}$ for $G$ which satisfies Desiderata \ref{LC} and \ref{KL}. Suppose that $\pi$ is an irreducible unitary $G$-representation and let $\phi = \LLC_{\mathbf{G}}(\pi)$. Then the element $\nu(\lambda_{\phi}) \in \mathfrak{h}_{\RR}^*$ satisfies
\begin{align}\label{eq FPP ineq}
    \langle \nu(\lambda_{\phi}), \alpha^{\vee}\rangle \leq 1,
\end{align}
for every simple co-root $\alpha^{\vee}$.
\end{thm}

\subsection{\texorpdfstring{The Levi subgroup $M_{\leq 1}$}{}}

Suppose $[\lambda]$ is (the equivalence class of) an infinitesimal character for ${}^L G$. In this subsection, we associate to $[\lambda]$ a standard Levi subgroup $ {}^LM_{\leq 1}([\lambda])$ of ${}^L G$. This Levi subgroup will play an important role in the proof of Theorem \ref{thm main}. 

Recall from \S \ref{sec notation} that the set of standard parabolic subgroups of ${}^LG$ is in natural bijection with the set of $\Gamma$-stable subsets of $\Delta(G^{\vee})$. Every standard parabolic subgroup ${}^L P $ of $^LG$ is a semidirect product $ P^{\vee} \rtimes \Gamma$, where $P^{\vee}$ is a standard parabolic subgroup of $G^{\vee}$. There is a unique Levi factor $M^{\vee}$ of $P^{\vee}$ containing $M_0^{\vee}$ and this subgroup is preserved by the action of $\Gamma$. We write $^LM := M^{\vee} \rtimes \Gamma$. A \emph{standard Levi subgroup} is any subgroup of $^LG$ arising in this fashion. Thus, we have a bijection
\begin{align*}
    \{\text{Standard Levi subgroups of }{}^LG\} &\to \{\Gamma \text{-stable subsets of }\Delta(G^{\vee})\},\\
    {}^L M & \mapsto \Delta(M^{\vee}),
\end{align*}
where $\Delta(M^{\vee})$ denotes the set of simple roots in $M^{\vee}$.

Now we regard $[\lambda]$ as an $L$-parameter for $^LG$, trivial on $\SL_2(\BC)$. Write $[\lambda] \leftrightarrow (P_{\lambda}=M_{\lambda}N_{\lambda}, [\lambda_t], \nu_{\lambda})$. 
Let ${}^L M_{\leq 1}([\lambda])$ be the standard Levi subgroup of ${}^L G$ corresponding to 
\begin{align}\label{eq def of Levi}
    \Delta( M_{\leq 1}^{\vee}([\lambda])):= \Delta( M_{\lambda}^{\vee}) \sqcup 
    \{\alpha \in \Delta(G^{\vee})\setminus \Delta(M_{\lambda}^{\vee})\ | \ \langle \nu_{\lambda}, \alpha^{\vee}\rangle \leq 1 \}.
\end{align}
Clearly, ${}^L M_{\leq 1}([\lambda])$ depends only on $M_{\lambda}$ and $\nu_{\lambda}$, and in fact only on $\nu_{\lambda}$ since ${}^L M_{\lambda}= Z_{{}^L G}( \nu_{\lambda})$ (Proposition \ref{prop SZ18}). Hence, we may write $ {}^L M_{\leq 1}(\nu_{\lambda}):= {}^L M_{\leq 1}([\lambda])$; when $\nu_{\lambda}$ is clear from the context, we will often simply write ${}^L M_{\leq 1}$. Also, we define
\[ [\lambda_{\leq 1}] := ({}^L M_{\lambda} \hookrightarrow {}^L M_{\leq 1}) \circ [\lambda_t^{z(\nu_{\lambda})}],\]
an infinitesimal character for ${}^L M_{\leq 1}$. {Recall from \S \ref{lc parameter} that $\lambda_t^{z(\nu_{\lambda})}$ is the $z(\nu_{\lambda})$-twist of $\lambda_t$: 
$\lambda_t^{z(\nu_{\lambda})}(w,x) = \lambda_t(w,x)z(\nu_{\lambda})^{d(w)},$
where $d(w)$ denotes the order of $w$.}

 \begin{remark}\label{rmk M leq 1}
     Regard $\nu_{\lambda}\in \mfr{a}_{M_{\lambda}}^{\ast}$ as an element of $ \Lie(Z({}^LM)^{\circ}_h) \subseteq  \mfr{g}^{\vee}$. Let $P_{\leq 1}^{\vee}= M_{\leq 1}^{\vee} N_{\leq 1}^{\vee}$ be the standard parabolic subgroup of $G^{\vee}$ and let 
     $\mfr{n}_{\leq 1}^{\vee}:= \Lie( N_{\leq 1}^{\vee})$. Then \eqref{eq def of Levi} implies that the eigenvalues of $\ad(\nu_{\lambda})|_{ \mfr{n}_{\leq 1}^{\vee} } $
     are all greater than $1$.
 \end{remark}

Recall from \S \ref{sec: KL} that to each infinitesimal character $\lambda$, we can associate a reductive group $H_{\lambda}$ and a vector space $V_{\lambda}$ with $H_{\lambda}$-action.

\begin{prop}\label{prop M_leq 1}
Let $\lambda$ be an infinitesimal character for ${}^LG$, and form the Levi subgroup ${}^L M_{\leq 1}$ of ${}^L G$ as above. Then the injections $V_{\lambda_{\leq 1}} \hookrightarrow V_{\lambda}$ and $H_{\lambda_{\leq 1}} \hookrightarrow H_{\lambda}$ are isomorphisms. 
\end{prop}
\begin{proof}
For a homomorphism $\Lambda: W_k \to {}^L G= G^{\vee} \rtimes \Gamma$, we let $\underline{\Lambda}: W_k \to G^{\vee}$ be the 1-cocycle defined by
\[ \Lambda(w)= (\underline{\Lambda}(w), w). \]
Suppose $[\lambda]  \leftrightarrow (P_{\lambda}, [\lambda_{t}], \nu_{\lambda})$. We can choose $\lambda_t$ so that 
\[ \lambda(w)= ( \underline{\lambda_{t}}(w) \cdot z(\nu_{\lambda})^{d(w)},w) \in M^{\vee}_{\lambda} \rtimes \Gamma. \]
Write $z= z(\nu_{\lambda})$.
For any standard Levi subgroup ${}^L M$ of ${}^L G$ containing ${}^L M_{\lambda}$ and $s \in \R$, we define 
\[  \mfr{m}^{\vee}:=\Lie({}^L M), \quad \mfr{m}_{z,s}^{\vee}:= \{x \in \mfr{m}^{\vee}\ | \ \Ad(z)x = s x \}.  \]
By definition, if ${}^L M_1 \leq {}^L M_2$, then $(\mfr{m}_1^{\vee})_{z,s} \subseteq (\mfr{m}_2^{\vee})_{z,s}$. On the other hand, by the definition of the Levi subgroup ${}^L M_{\leq 1}$, the inclusion is an equality when ${}^L M_1= M_{\leq 1}, {}^L M_2=G$ and $1 \leq s \leq q$:
\begin{align}\label{eq Lie t eigenspace}
    (\mfr{m}^{\vee}_{\leq 1})_{z,s}=\mfr{g}_{z,s}^{\vee}, \forall \, 1 \leq s \leq q.
\end{align}
We claim that 
\begin{enumerate}
    \item [(i)] $V_{\lambda}=  \{x \in \mfr{g}^{\vee}_{z, q}\ | \ \Ad(\lambda(w))x= |w|x,\ \forall w \in W_k\}, $ and,
    \item [(ii)] $\Lie(H_{\lambda})= \{ x \in \mfr{g}_{z,1}^{\vee}\ | \  \Ad(\lambda(w))x=x ,\ \forall w \in W_k\}$;
\end{enumerate}
 and similarly,
 \begin{enumerate}
    \item [(i$'$)] $V_{\lambda_{\leq 1}}=  \{x \in (\mfr{m}^{\vee}_{\leq 1})_{z, q}\ | \ \Ad(\lambda_{\leq 1}(w))x= |w|x,\ \forall w \in W_k\}, $ and 
    \item [(ii$'$)] $\Lie(H_{\lambda_{\leq 1}})= \{ x \in (\mfr{m}^{\vee}_{\leq 1})_{z,1}\ | \  \Ad(\lambda_{\leq 1}(w))x=x ,\ \forall w \in W_k\}$.
\end{enumerate}
We verify Claim (i) below; the other claims are proved analogously.

Recall that by definition
\[V_{\lambda} =  \{x \in \mfr{g}^{\vee}\ | \ \Ad(\lambda(w))x= |w|x,\ \forall w \in W_k\}.\]
Hence, $V_{\lambda_{\leq 1}} \supset  \{x \in (\mfr{m}^{\vee}_{\leq 1})_{z, q}\ | \ \Ad(\lambda_{\leq 1}(w))x= |w|x,\ \forall w \in W_k\}$. 
Now take any $x \in V_{\lambda}$, we will show that $x \in \mfr{g}^{\vee}_{z,q}$. Let $K$ be the splitting field of $G$ so that the action of $\Gamma$ on $G^{\vee}$ factors through the finite quotient $\Gamma_K:= \Gal(K/k)$, and let $r:= |\Gal(K/k)|$. Then $w^r$ acts trivially on $G^{\vee}$ and $\mfr{g}^{\vee}$ for any $ w \in W_k$, and we have that
\begin{align*}
    |w|^r x= \Ad(\lambda(w^r))x= \Ad( \underline{\lambda_{t}}(w^r) \cdot z^{r \cdot d(w)}, w^r)x = \Ad( \underline{\lambda_{t}}(w^r) \cdot z^{r \cdot d(w)})x.
\end{align*}
Since $z$ commutes with $\underline{\lambda_{t}}(w^r)$ and both are  semisimple, $x$ must be an eigenvector for both $ \Ad( \underline{\lambda_{t}}(w^r))$ and  $\Ad(z^{r\cdot d(w)})$. Since $z$ is hyperbolic and $\underline{\lambda_{t}}(w^r)$ is elliptic (see \cite[Propositions 5.3, 6.1]{SZ18}), we must have that
\[ \Ad( \underline{\lambda_{t}}(w^r)) x =x , \quad  \Ad(z^{d(w)})x =|w|x.\]
Therefore, $x$ must lie in $\mfr{g}_{z,q}^{\vee}$, which completes the verification of Claim(i). 

Now since $\lambda= ({}^L M_{\leq 1} \hookrightarrow {}^L G)\circ \lambda_{\leq 1}$, Claims (i), (ii), (i$'$), (ii$'$), and \eqref{eq Lie t eigenspace} imply that $V_{\lambda_{\leq 1}}= V_{\lambda}$ and the identity components of $H_{\lambda_{\leq 1}}$ and $H_{\lambda}$ coincide. Finally, since ${}^L M_{\leq 1}$ contains the Levi subgroup ${}^L M_{\lambda}$, by Theorem \ref{thm:Langlandsparameters}(ii), the component groups of $H_{\lambda_{\leq 1}}$ and $H_{\lambda}$, which are $\pi_0(\lambda_{\leq 1})$ and $\pi_{0}(\lambda)$, also coincide. We conclude that the map $H_{\lambda_{\leq 1}} \hookrightarrow H_{\lambda}$ is an isomorphism. This completes the proof of the proposition.
\end{proof}

\begin{lemma}\label{lem Levi LC containment}
Let $[\lambda]$ be an infinitesimal character for ${}^L G$ and let $[\phi]\in \Phi_{\lambda}({}^L G)$.
    Suppose that $[\phi] \leftrightarrow (P=MN, [\phi_t], \nu_{\phi})$. Then $
    {}^L M$ is contained in ${}^L M_{\leq 1}([\lambda])$.
\end{lemma}
\begin{proof}
Write $[\lambda] \leftrightarrow (P_{\lambda}, [\lambda_t], \nu_{\lambda})$. Take a representative $\lambda_{\ast} = ({}^LM_{\lambda} \hookrightarrow {}^L G) \circ \lambda_{\ast,t}^{z(\nu_{\lambda})}$ of $[\lambda]$.
 Recall that $z(\lambda):= z(\nu_{\lambda})\in Z({}^L M_{\lambda})^{0} \subseteq {T}_0^{\vee}$, where $\mathbf{T}_0$ is the maximal torus  of the minimal $k$-parabolic subgroup $\mathbf{P}_0$. As $\nu_{\lambda}$ is positive with respect to $P_{\lambda}$,  it is $\Delta(G^{\vee})$-dominant. Namely, for any $\alpha^{\vee} \in \Delta(G^{\vee})$, we have $ \langle \nu_{\lambda}, \alpha^{\vee}\rangle \geq 0$. Moreover, we have (Proposition \ref{prop SZ18})
\begin{align}\label{eq centralizer z}
    {}^L M_{\lambda}= Z_{{}^L G}(z(\lambda)).
\end{align}

Take any $e_{\ast} \in \OO_{\phi}$, the $H_{\lambda_{\ast}}$-orbit in $V_{\lambda_{\ast}}$ corresponding to $[\phi]$. According to \cite[Lemma 2.1]{GR10}, we can complete $e_{\ast}$ into an $\mathfrak{sl}_2$-triple $(e_{\ast},f_{\ast},h_{\ast})$, such that 
\[\Ad(\lambda_{\ast}(w))h_{\ast}=h_{\ast},\ \  \Ad(\lambda_{\ast}(w))f_{\ast}= |w|^{-1} f_{\ast}, \ \ \forall w \in W_k.\]
In particular, $\Ad(z(\lambda))h_{\ast}=h_{\ast}$, so \eqref{eq centralizer z} implies that $ h_{\ast} \in \Lie(M^{\vee}_{\lambda})$. We may take a $g \in M_{\lambda}^{\vee}$ such that $\Ad(g)h_{\ast} \in \Lie({T}_0^{\vee})$ and set $(e,f,h):= \Ad(g) (e_*,f_*,h_*)$. Let $\lambda_t:= g \lambda_{{\ast},t} g^{-1} \in [\lambda_t]$ and  $\lambda:= g \lambda_{\ast} g^{-1}\in [\lambda]$. Then $\lambda \leftrightarrow (P_{\lambda}, \lambda_t, \nu_{\lambda})$, and $e \in V_{\lambda}$. Take $\varphi: \SL_2(\BC) \to G^{\vee}$ such that 
\[ d\varphi\begin{pmatrix}
    0 & 1 \\0 & 0
\end{pmatrix}= e,\ \ d\varphi\begin{pmatrix}
    1 & 0 \\0 & -1
\end{pmatrix}= h,\ \ d\varphi\begin{pmatrix}
    0 & 0 \\1 & 0
\end{pmatrix}= f, \]
and define
\[ \phi( w,x):= \lambda(w) \cdot \varphi\begin{pmatrix}
    |w|^{-1/2} & \\ & |w|^{1/2}
\end{pmatrix} \cdot \varphi(x).\]
Then $\phi \in [\phi]$ and $\lambda_{\phi}=\lambda$.

Now we consider $\phi|_{W_k}$, which is an $L$-parameter, trivial on $\SL_2(\BC)$, and define 
\[z_{\phi}:=  z(\phi|_{W_k}) ,\ \ z_{\SL_2}:= \phi|_{\SL_2(\BC)}\left(\begin{pmatrix}
    q^{1/2} & \\ & q^{-1/2}
\end{pmatrix} \right). \]
Clearly, we have
\[ z(\lambda)= z_{\phi} \cdot z_{\SL_2}.  \]
Let $ ({}^LM_{\phi}, {}^L N_{\phi} z_{\phi})$ be the triple defined in Proposition \ref{prop SZ18}, which is conjugate to $({}^LM, {}^LN, z(\nu))$ by some $g^{\vee} \in G^{\vee}$. To show that ${}^L M$ is contained in $ {}^L M_{\leq 1}$, it suffices to show that 
\[ ^L M_{\phi} \subseteq  {}^L M_{\leq 1}, \text{ and }{}^LN_{\leq 1} \subseteq {}^L N_{\phi} \subseteq {}^L P_{\leq 1},\]
which are consequences of the following claim. 

\begin{claim}
    For any positive root $\gamma^{\vee}$ in $\mfr{n}_{\leq 1}^{\vee} = \mathrm{Lie}(N_{\leq 1}^{\vee})$, we have
\[ \langle \nu_{\lambda}- \half{h} , \gamma^{\vee} \rangle >1.\]
\end{claim}

Now we prove the claim. We regard $\nu_\lambda$ as an element in $\Lie(T_0^{\vee})_{h}.$ Then it suffices to show that any eigenvalue of $\ad(\nu_{\lambda} -\half{h})|_{\mfr{n}^{\vee}_{\leq 1}} \in \mathrm{End}(\mfr{n}_{\leq 1}^{\vee})$ is greater than 1.

Note that $(e,f,h)$ is contained in $\mfr{m}_{\leq 1}^{\vee}$. Thus, $d{\varphi}(\mathfrak{sl}_2) \subset \mfr{m}_{\leq 1}^{\vee}$. Regard $\mathfrak{n}^{\vee}_{\leq 1}$ as an $\mathfrak{sl}_2(\mathbb{C})$-representation via the homomorphism $\ad \circ d\varphi$, and consider its decomposition into isotypic components:
 \[\mfr{n}^{\vee}_{\leq 1} = \bigoplus_{d>0} V_{d}.\]
Here $V_d$ is the sum of all $d$-dimensional irreducible subrepresentations of $\mfr{n}^{\vee}_{\leq 1}$. Let 
\[ V_d^{h=t}:= \{X \in V_d \ | \ [h,X]=tX\}.\]
Then we may further decompose
\[ V_d= V_{d}^{h={1-d}} \oplus V_{d}^{h={3-d}} \oplus \cdots \oplus V_{d}^{h={d-1}},\]
and $ \ad(e)|_{V_{d}^{h=t}}:V_{d}^{h=t} \to V_{d}^{h=t+2}$
is an isomorphism for $t= 1-d,\ldots, d-3$.

Now note that $[\nu_{\lambda}- \half{h}, \half{h} ]=0$ since $\nu_{\lambda}, h \in \Lie({T}_0^{\vee})$, and
\[ [ \nu_{\lambda}-\half{h} , e]= [\nu_{\lambda}, e]- [\half{h},e]= e-e=0 .\]
Thus, $ \ad( \nu_{\lambda}- \half{h})|_{\mfr{n}^{\vee}_{\leq 1}}$ commutes with the $\mfr{sl}_2$-homomorphism $(\ad \circ d{\varphi})|_{\mfr{n}^{\vee}_{\leq 1}}$. Hence, $\ad(\nu_{\lambda}- \half{h} )|_{\mfr{n}^{\vee}_{\leq 1}}$ preserves $ V_{d}^{h=t}$ for any $t$. It remains to show that every eigenvalue of $\ad(\nu_{\lambda}- \half{h} )|_{V_{d}^{h=t}}$ is larger than $1$ for $ t= 1-d, 3-d, \cdots, d-1$. 

Let $X$ be an eigenvector of $\ad(\nu_{\lambda}- \half{h} )$ in $V_{d}^{h=t}$ for a $ t\in \{1-d, 3-d, \ldots, d-1\}$. We take $\{X_{1-d}, \ldots, X_{d-1}\}$ such that $X_t= X$ and
\[ X_{s+1}= [e,X_{s}] \]
for $s= 1-d,3-d,\ldots, d-3$. Write 
\[[\nu_{\lambda},X_{1-d}]= c' X_{1-d}. \]
Since $V_{d}^{h=1-d} \subseteq \mfr{n}_{\leq 1}^{\vee}$, we have $c'>1$ by the definition of the Levi subgroup ${}^LM_{\leq 1}$ (see Remark \ref{rmk M leq 1}). By induction on $s$, we deduce that
\[ [ \nu_{\lambda} - \half{h},X_s ]= (c'- \half{1-d}) X_s  \]
for $s= 1-d, 3-d,\ldots, d-1$. 
The case when $s=1-d$ is clear. For $s>1-d$, we have
\[ [ \nu_{\lambda} - \half{h},X_s ]= [ \nu_{\lambda} - \half{h},[ e,X_{s-1}] ]=[e, [ \nu_{\lambda} - \half{h},X_{s-1} ]]= (c'- \half{1-d}) X_s \]
by the induction hypothesis. Since $c'- \half{1-d}>1$, this completes the proof of the claim and the lemma.
\end{proof}

Here is the main theorem of this subsection.

\begin{thm}\label{thm ind irred}
    Assume that there exists a local Langlands correspondence satisfying Desiderata \ref{LC}, \ref{KL}. Let $\lambda$ be an infinitesimal character for ${}^L G$, and assume that $\Pi_{\lambda}(G)$ is non-empty. Then the following are true:
    \begin{enumerate}
        \item [(a)] The parabolic subgroup ${}^L P_{\leq 1}$ is $G$-relevant. Let $P_{\leq 1} = M_{\leq 1}N_{\leq 1}$ be the corresponding parabolic subgroup of $G$. 
        \item [(b)] The map
        \begin{align*}
          \Ind_{P_{\leq 1}}^G:  \Pi_{\lambda_{\leq 1}}(M_{\leq 1}) &\to \Pi_{\lambda}(G)
        \end{align*}
        is a well-defined bijection. That is, for every $\pi \in \Pi_{\lambda_{\leq 1}}(M_{\leq 1})$, the induction $\Ind^G_{P_{\leq 1}} \pi$ is irreducible, and every irreducible in $\Pi_{\lambda}(G)$ can be obtained in this fashion.
    \end{enumerate}
\end{thm}
\begin{proof}
    For Part (a), Proposition \ref{prop M_leq 1} implies that the map 
\[ V_{\lambda_{\leq 1}}/ H_{\lambda_{\leq 1}} \to V_{\lambda}/ H_{\lambda}\]
induced from the embedding ${}^L M_{\leq 1} \hookrightarrow {}^L G$ is a bijection. Thus, the map of $L$-parameters
\begin{align*}
  \iota: \Phi_{\lambda_{\leq 1}}(M_{\leq 1}) &\to  \Phi_{\lambda}({}^LG),\\
    [\phi_{\leq 1}] & \mapsto ({}^L M_{\leq 1} \to {}^L G) \circ [\phi_{\leq 1}].
\end{align*}
is a bijection as well. Since $\Pi_{\lambda}(G)$ is non-empty, there must exist a $[\phi] \in \Phi_{\lambda}({}^L G)$ which is $G$-relevant, see Lemma \ref{lem:relevant}. Since any $L$-parameter in $\Phi_{\lambda}({}^L G)$ must factor through ${}^L M_{\leq 1}$, we conclude that ${}^L M_{\leq 1}$ is $G$-relevant. This proves Part (a).

Now we prove Part (b). For any $[\phi] \in \Phi_{\lambda}({}^L G)$, write $[\phi_{\leq 1}]:= \iota^{-1}([\phi])$.
The embedding of $L$-groups induces a homomorphism
\[S_{[\phi_{\leq 1}]} \to S_{[\phi]}.\]
According to the isomorphism \eqref{eq maps of component groups}, Proposition \ref{prop M_leq 1} further implies that the above map is an isomorphism. Thus, we have a bijection
\begin{align*}
    \Pi_{\phi_{\leq 1}}(M_{\leq 1}) &\to \Pi_{\phi}(G),\\
    \pi(\phi_{\leq 1}, \tau) & \mapsto \pi(\phi, \tau),
\end{align*}
where $\tau $ is an irreducible representation of $S_{[\phi_{\leq 1}]} \cong S_{[\phi]} $.

We first claim that induction induces a bijection on the level of standard modules:
\begin{align}\label{eq thm ind std}
    \Ind_{M_{\leq 1}}^G (M([\phi_{\leq 1}], \tau))= M( [\phi], \tau)
\end{align}
Indeed, let $[\phi] \leftrightarrow (P_{\phi}, [\phi_t], \nu_{\phi})$. Lemma \ref{lem Levi LC containment} implies that $M_{\phi} \subseteq  M_{\leq 1}$. Thus, $(P_{\phi}\cap M_{\leq 1}, [\phi_t], \nu_{\phi})$ is a standard triple for ${}^L {M_{\leq 1}}$. The $L$-parameter 
\[ [\phi_{\leq 1}']:= ({}^L M_{\phi} \hookrightarrow {}^L M_{\leq 1}) \circ [\phi_t^{z(\nu_{\phi})}]\]
satisfies $\iota([\phi_{\leq 1}'])= [\phi]$. Thus, the injectivity of $\iota$ implies that $[\phi_{\leq 1}']=[\phi_{\leq 1}]$. We conclude that  $ [\phi_{\leq 1}]\leftrightarrow (P_{\phi}\cap M_{\leq 1}, [\phi_t], \nu_{\phi})$. Now the claim is a direct consequence of Desideratum \ref{LC} and the transitivity of induction:
\begin{align*}
    \Ind_{M_{\leq 1}}^G (M([\phi_{\leq 1}], \tau))=  \Ind_{M_{\leq 1}}^G ( \Ind_{M_{\phi}}^{M_{\leq 1}} (\pi([\phi_{t}], \tau) \otimes \nu_{\phi}))=\Ind_{M_{\phi}}^G ( \pi([\phi_{t}], \tau) \otimes \nu_{\phi})= M([\phi],\tau).
\end{align*}

To complete the proof of (b), it suffices to show that 
\[ \Ind_{M_{\leq 1}}^G \pi(\phi_{\leq 1}, \tau)= \pi(\phi, \tau).\]
By Desideratum \ref{KL} and Proposition \ref{prop M_leq 1}, we have
\begin{align*}
   m_{r} ( ([\phi], \tau), ([\phi'],\tau'))&= (-1)^{\dim(\OO_{\phi})- \dim(\OO_{\phi'})} m_g( ([\phi'],\tau'), ([\phi_{\leq 1}],\tau))\\
    &= (-1)^{\dim(\OO_{\phi_{\leq 1}})- \dim(\OO_{\phi'_{\leq 1}})} m_g( ([\phi_{\leq 1}'],\tau'), ([\phi_{\leq 1}],\tau))\\
    &=   m_{r} ( ([\phi_{\leq 1}], \tau), ([\phi_{\leq 1}'],\tau')).
\end{align*}
Let $m_{r}^{-1}$ denote the inverse of matrix of $m_r$. Then by the claim \eqref{eq thm ind std}, we obtain that
\begin{align*}
    \Ind_{M_{\leq 1}}^{G} (\pi(\phi_{\leq 1}, \tau))&= \Ind_{M_{\leq 1}}^G \left( \sum_{([\phi_{\leq 1}'],\tau') \in \Phi^{\mfr{e}}_{\lambda_{\leq 1}}( {}^L M_{\leq 1})}m_{r}^{-1} (([\phi_{\leq 1}],\tau), ([\phi_{\leq 1}'],\tau')) M([\phi_{\leq 1}'],\tau') \right)\\
    &=\sum_{([\phi_{\leq 1}'],\tau') \in \Phi^{\mfr{e}}_{\lambda_{\leq 1}}( {}^L M_{\leq 1})}m_{r}^{-1} (([\phi_{\leq 1}],\tau), ([\phi_{\leq 1}'],\tau'))  \Ind_{M_{\leq 1}}^{G}(M([\phi_{\leq 1}'],\tau'))\\
    &= \sum_{([\phi'],\tau') \in \Phi^{\mfr{e}}_{\lambda}( {}^L G)}m_{r}^{-1} (([\phi],\tau), ([\phi'],\tau'))  M([\phi'],\tau')\\
    &= \pi( [\phi], \tau).
\end{align*}
This completes the proof of Part (b).
\end{proof}

\subsection{Proof of Theorem \ref{thm main}}

We begin with two simple lemmas. Suppose that $\pi \leftrightarrow (P,\pi_t,\nu)$.
The first lemma gives necessary and sufficient conditions on $(P,\pi_t,\nu)$ for the representation $\pi$ to be Hermitian.

\begin{lemma}[{\cite[ \S 5.1]{BM89}}]\label{lem Hermitian}
    Suppose that $\pi \leftrightarrow (P, \pi_t, \nu)$. Then $\pi$ is Hermitian if and only if there exists an element $w \in W(G,A) := N_G(A)/Z_G(A)$ such that $w(M)=M$, $\pi_t^w \cong \pi_t$ and $w(\nu)=-\nu$.
\end{lemma}

The second lemma gives a sufficient condition for non-unitarity. It is a consequence of \cite[Theorem 2.5]{Tad88} combined with the discussion in \cite[\S 3(b)]{Tad93}.

\begin{lemma}\label{lem top unitary dual}
Let $P=MN$ be a parabolic subgroup of $G$. Let $\pi$ be an irreducible representation of $M$ and let $I$ be an unbounded connected subset of $\mathfrak{a}_M^* = X^{\ast}(M)\otimes_{\Z} \R$. We identify each $s \in I$ with a positive real-valued character $\nu_s$ of $M$ via the identification (\ref{eq:astar}). Suppose that for any $s \in I$, the induced representation $\pi_s:=\Ind_{P}^G (\pi \otimes \nu_s)$ is irreducible and Hermitian. Then $ \pi_s$ is non-unitary for every $s \in I$.
\end{lemma}

\begin{proof}[Proof of Theorem \ref{thm main}]
Let $\pi$ be a Hermitian representation of $G$. Write $\pi \leftrightarrow (P=MN, \pi_{t}, \nu)$ and choose a Weyl group element $w \in W(G,A)$ such that $w(M)=M$, $w(\pi_t) \cong \pi_t$ and $w(\nu)=-\nu$ by Lemma \ref{lem Hermitian}.

Write $[\lambda_{\LLC(\pi)}] \leftrightarrow (P_{\lambda},[\lambda_{t}], \nu)$. Suppose that the set 
\[ S:=\{ \alpha^{\vee} \in \Delta(G^{\vee}) \ | \  \langle \nu, \alpha^{\vee}\rangle >1  \} \]
is non-empty. We will show in this case that $\pi$ cannot be unitary.

For $s \geq 0$, take $\nu_s \in X^{\ast}(M) \otimes_Z \R$ such that for any simple co-root $\beta^{\vee} \in \Delta(G^{\vee})$,
\begin{align*}
    \langle \nu_s , \beta^{\vee} \rangle= \begin{cases}
        s & \text{ if }  \beta^{\vee} \in S, \\
        0 & \text{otherwise.}
    \end{cases}
\end{align*}
Define the infinitesimal character $\lambda_s:=({}^L M \hookrightarrow {}^L G)\circ \lambda_t^{z(\nu_s)}$. Let $M_{\leq 1} ([\lambda_s])$ be the (proper) Levi subgroup of $G$ defined in the previous subsection for $\lambda_s$. It follows from the definition of $\nu_s$ that the Levi subgroup $M_{\leq 1}(\lambda_s)$ is independent of $s $ for any $s \geq 0$.

 By Theorem \ref{thm ind irred}(b), there exists an irreducible representation $\pi_{\leq 1} \in \Pi_{\lambda_{\leq 1}}(M_{\leq 1}) $ such that
\[ \pi = \Ind_{M_{\leq 1}}^{G} \pi_{\leq 1}.\]
Now we regard $\nu_s$ as an unramified character of $M_{\leq 1}$, and consider the representation $\pi_{\leq 1}\otimes \nu_s$ of $M_{\leq 1}$. By Lemma \ref{lem Levi LC containment} and Desideratum \ref{LC}, the Levi subgroup $M$ is contained in $M_{\leq 1}$. Thus, we may regard $\nu_s$ as an unramified character of $M$ via restriction.

We claim that the following hold for any $s \geq 0$:
\begin{enumerate}
    \item [(i)] $ \pi_{\leq 1}\otimes \nu_s\leftrightarrow (P \cap M_{\leq 1}, \pi_t, \nu+ \nu_s)$;
    \item [(ii)] $\pi_s:=\Ind_{M_{\leq 1}}^G (\pi_{\leq 1}\otimes \nu_s) $ is irreducible, and $\pi_s \leftrightarrow (P,\pi_t, \nu+ \nu_s)$;
    \item [(iii)] $\pi_s$ is Hermitian.
\end{enumerate}

These claims imply that $\pi_s$ is not unitary for every $s \geq 0$ (and hence that $\pi=\pi_0$ is not unitary) by Lemma \ref{lem top unitary dual}.

Now we prove the claims. For Claim (i), first observe that $\nu+\nu_s$ is $\Delta(G^{\vee})$-dominant. Hence, $(P \cap M_{\leq 1}, \pi_t, \nu+ \nu_s) $ is a standard triple for ${}^L M_{\leq 1}$. Next, we have
 \[ \Ind_{M}^{M_{\leq 1}}( (\pi_t \otimes \nu) \otimes \nu_s) \cong \left(\Ind_{M}^{M_{\leq 1}}( \pi_t \otimes \nu)\right) \otimes \nu_s. \]
 as finite-length representations. Recall from the proof of Theorem \ref{thm ind irred} that $\pi_{\leq 1} \leftrightarrow (P \cap M_{\leq 1}, \pi_t, \nu)$. Namely, 
 \[  \Ind_{M}^{M_{\leq 1}} (\pi_t \otimes \nu )\twoheadrightarrow \pi_{\leq 1}.\]
Then tensoring the above surjection with $\nu_s$ gives what we want.

Now we prove Claim (ii). Claim (i) and Desideratum \ref{LC} imply that
\[  ({}^L M_{\leq 1} \hookrightarrow {}^L G) \circ [\lambda_{\LLC(\pi_{\leq 1}\otimes \nu_s)}]  =[\lambda_s]. \]
Thus, the irreducibility of the induction follows from Theorem \ref{thm ind irred}(b). Also, part (i) implies that
\[ \Ind_{M}^{M_{\leq 1}}(\pi_t \otimes (\nu+\nu_s)) \twoheadrightarrow \pi_{\leq 1} \otimes \nu_s. \]
The exactness of parabolic induction implies that
\[\Ind_{M}^{G} (\pi_t \otimes (\nu+\nu_s)) =  \Ind_{M_{\leq 1}}^{G} \circ \Ind_{M}^{M_{\leq 1}}(\pi_t \otimes (\nu+\nu_s)) \twoheadrightarrow  \Ind_{M_{\leq 1}}^G (\pi_{\leq 1} \otimes \nu_s)= \pi_s. \]
This implies that $\pi_s \leftrightarrow (P,\pi_t, \nu+ \nu_s)$ and completes the verification of Claim (ii).

Finally, we prove Claim (iii). Recall that we have chosen an element $w \in W(G,A)$ such that $w(M)=M$, $\pi_t^w \cong \pi_t$ and $w(\nu)=-\nu$. Since $\pi_s \leftrightarrow (P,\pi_t, \nu+ \nu_s)$, by Lemma \ref{lem Hermitian}, it suffices to show that $w(\nu_s)=-\nu_s$. 
Observe that for any coroot $\gamma^{\vee} \in R(G^{\vee})$, we have
\[ \langle \nu , w^{-1}(\gamma^{\vee}) \rangle= \langle w(\nu) , \gamma^{\vee} \rangle=\langle -\nu , \gamma^{\vee} \rangle = - \langle \nu , \gamma^{\vee} \rangle.\]
Thus, the Weyl group element $w^{-1}$ permutes the coroots in $R(M^{\vee})$, sends positive simple coroots in $\Delta(G^{\vee}) \setminus \Delta(M^{\vee})$ to negative simple coroots, and sends $S$ to $-S$, where we recall that 
\[ S=\{ \alpha^{\vee} \in \Delta(G^{\vee}) \ | \  \langle \nu, \alpha^{\vee}\rangle >1  \}.\]
Thus, for any simple coroot $\beta^{\vee}$ of $G^{\vee}$, we have
\[ \langle w(\nu_s), \beta^{\vee}\rangle=\langle \nu_s, w^{-1}(\beta^{\vee})\rangle= \langle - \nu_s, \beta^\vee\rangle.\]
We conclude that $w(\nu_s)=-\nu_s$. This completes the proof of the claims and the theorem.
\end{proof}

\section{Some remarks on the sharpness of the FPP bound}

For a quasi-split group $G$, the FPP bound on $\Pi_u(G)$ is sharp in the following sense: there exists a representation in $\Pi^{\textrm{pure}}(G)$ such that \emph{all} of inequalities in \eqref{eq FPP ineq} are equalities (i.e. $\nu(\lambda_\phi)$ is equal to the half-sum of the positive roots for $G$). Indeed, this is the case for the trivial and Steinberg representations.

On the other hand, if one restricts to representations belonging to a particular Bernstein component, the FPP bound may not be sharp in the sense described above. In some special cases, we may apply the idea of unramification in \cite[Chapter 5]{CFMMX22} to improve the unitarity bound. We elaborate on this idea in the following example; we will pursue it further in future work.

\begin{exmp}\label{exmp FPP not sharp}
Let $\rho$  and $\tau$ be two non-isomorphic irreducible unitary supercuspidal representations of $\GL_d(F)$ with trivial central character. Let $P$ be the standard parabolic subgroup of $\GL_{3d}(F)$ whose Levi subgroup is isomorphic to $\GL_d(F) \times \GL_d(F) \times \GL_d(F)$. For $a,b_1, b_2 \in \R$, let $\pi_{a,b_1, b_2}$ be the unique irreducible quotient of $ \Ind_{P}^{\GL_{3d}} (\rho \lvert\cdot\rvert^{a} \otimes  \tau \lvert\cdot\rvert^{b_1} \otimes  \tau \lvert\cdot\rvert^{b_2})$. Note that the induction is irreducible unless $|b_1-b_2|=1$. In order that $\pi_{a, b_1, b_2}$ is Hermitian, we must have $a=0$ and $b_1= -b_2$. Thus, we write $\pi_{b}:= \pi_{0, b, -b}$ for short.
The $\fpp$ bound suggests that $\pi_{b}$ is unitary only if $ |b| \leq 1$. However, according to \cite{Tad86}, $\pi_{b}$ is unitary if and only if $|b| \leq  \half{1}$.

Now we compute the unramification process from \cite[Chapter 5]{CFMMX22} explicitly. For $\sigma \in \{\rho,\tau\}$, write the $L$-parameter of $\sigma$ as $\phi_{\sigma}(w,x)= ( \underline{\phi_{\sigma}}(w), w)$ under the local Langlands correspondence of $\GL_d(F)$. Then, the $L$-parameter of $\pi_{b}$ is $\phi_{b}:= \lambda_{b}:= \phi_{\tau} \lvert\cdot\rvert^{b}\oplus  \phi_{\rho} \oplus \phi_{\tau} \lvert\cdot\rvert^{-b} $, which is trivial on $\SL_2(\BC)$. 

    Fixing any choice of Frobenius $\Fr$, clearly, the elliptic part and the hyperbolic part of $\lambda_{b}(\Fr)$ are $t_{\lambda}\rtimes \Fr= ( \underline{\phi_{\tau}}(\Fr) \oplus \underline{\phi_{\rho}}(\Fr)\oplus \underline{\phi_{\tau}}(\Fr) ) \rtimes \Fr $ and $s_{\lambda} \rtimes 1=  \textrm{diag}(q^{b} I_{d} , I_d, q^{-b} I_{d}) \rtimes 1 $ respectively. Thus, by Schur's lemma, we have
    \[ J_{\lambda}:= Z_{\GL_{3d}(\BC)} (\underline{\phi_{\tau}} \oplus \underline{\phi_{\rho}}\oplus \underline{\phi_{\tau}}) = \left\{ \begin{pmatrix}
        \alpha I_d &  &  \beta I_d\\& \epsilon I_d& \\ \gamma I_d&& \delta I_d
    \end{pmatrix} \ \middle| \det\begin{pmatrix}
        \alpha & \beta \\ \gamma & \delta 
    \end{pmatrix} \neq 0, \epsilon \neq 0\ \right\} \cong \GL_1(\BC) \times \GL_2(\BC). \]
    Thus, the unramification process gives the split group
    $G_{\lambda_{b}}(F) \cong \GL_1(F) \times \GL_2(F) $, and the unramified infinitesimal character $(\lambda_{b})_{\nr}$ is given by
    \[(\lambda_{b})_{\nr}(w)= \left(\left( 1, \begin{pmatrix}
        |w|^{b} & \\ & |w|^{-b}
    \end{pmatrix}\right),w\right). \]
    The $\fpp$ bound for the pair $((\lambda_{b})_{\nr}, G_{\lambda} )$ suggests a better bound $|b| \leq \half{1}$. Indeed, following the same idea of the proof of Theorem \ref{thm main}, the family of parabolic induction 
    \[\left\{\pi_b =\Ind_{P}^{\GL_{3d}} (\rho \otimes  \tau \lvert\cdot\rvert^{b} \otimes  \tau \lvert\cdot\rvert^{-b})\ \middle| \ b \in \left(\half{1}, \infty\right)\right\}\]
    is irreducible, Hermitian and unbounded. Hence, they must be non-unitary.
\end{exmp}

\end{document}